\documentclass{amsart}
\usepackage{amsfonts}
\usepackage{amsmath}
\usepackage{amssymb}
\usepackage{amsthm}
\usepackage{graphicx}
\usepackage{capt-of}
\usepackage{tikz}
\usetikzlibrary{patterns}
\usetikzlibrary{decorations.markings}
\usepackage{multirow}
\usepackage{array} 
\newtheorem{theorem}{Theorem}
\newtheorem{prop}{Proposition}
\newtheorem{lemma}{Lemma}
\newtheorem{rem}{Remark}
\newtheorem{exmp}{Example}

\begin{document}

\title{$Z$-knotted triangulations of surfaces}
\author{Mark Pankov, Adam Tyc}
\subjclass[2000]{}
\keywords{embeded graph, triangulation, zigzag}
\address{Mark Pankov: Faculty of Mathematics and Computer Science, 
University of Warmia and Mazury, S{\l}oneczna 54, 10-710 Olsztyn, Poland}
\email{pankov@matman.uwm.edu.pl}

\address{Adam Tyc: Institute of Mathematics, Polish Academy of Science, \'Sniadeckich 8, 00-656 Warszawa, Poland}
\email{atyc@impan.pl}

\maketitle

\begin{abstract}
A zigzag in a map (a $2$-cell embedding of a connected graph in a connected closed $2$-dimensional surface) 
is a cyclic sequence of edges satisfying the following conditions: 
1) any two consecutive edges lie on the same face and have a common vertex,
2) for any three consecutive edges the first and the third edges are disjoint
and the face containing the first and the second edges is distinct from the face which contains the second and the third.
A map is $z$-knotted if it contains a single zigzag.
Such maps are closely connected to Gauss code problem and have nice homological properties.
We show that every triangulation of a connected closed $2$-dimensional surface admits a $z$-knotted shredding.
\end{abstract}

\section{Introduction}
Consider two adjacent edges $e_{1}$ and $e_{2}$ in a regular polyhedron.
These edges have a common vertex $v$ and lie on the same face.
We take the second face containing $e_{2}$. 
In this face, there is a unique edge $e_{3}$ intersecting $e_{2}$ in a vertex different from $v$.
Similarly, we get  an edge $e_{4}$ from the edges $e_{2}$ and $e_{3}$.
Step by step, we create a sequence $e_{1},e_{2},e_{3},\dots$ whose elements will repeat  after some times.  
Since we are dealing with a regular polyhedron, this is a skew polygon without self-intersections.
Coxeter called such polygons {\it Petrie polygons} \cite[Section 2.6]{Coxeter}. They play an important role in his famous  book. 
See Figure 1 for the Petrie polygons on  the five Platonic solids (the bold line),
observe that the dual solids have the same Petrie polygons. 
%%%%%%%%%%%%%PlatonS%%%%%%%%%%%%%
\begin{center}
\begin{tikzpicture} [scale=0.6]
\draw[xshift=4.375cm, fill=black] (0:1.75cm) circle (3pt);
\draw[xshift=4.375cm, fill=black] (90:1.75cm) circle (3pt);
\draw[xshift=4.375cm, fill=black] (180:1.75cm) circle (3pt);
\draw[xshift=4.375cm, fill=black] (270:1.75cm) circle (3pt);
    \draw[xshift=4.375cm,thick,line width=2pt] (0:1.75cm) \foreach \x in {90,180,...,359} {
            -- (\x:1.75cm) 
        } -- cycle (90:1.75cm);
    \draw[xshift=4.375cm, dashed] (0:1.75cm) \foreach \x in {90,270} {-- (\x:1.75cm)};
    \draw[xshift=4.375cm] (0:1.75cm) \foreach \x in {0,180} {-- (\x:1.75cm)};

\draw[xshift=8.75cm, fill=black] (0:1.75cm) circle (3pt);
\draw[xshift=8.75cm, fill=black] (60:1.75cm) circle (3pt);
\draw[xshift=8.75cm, fill=black] (120:1.75cm) circle (3pt);
\draw[xshift=8.75cm, fill=black] (180:1.75cm) circle (3pt);
\draw[xshift=8.75cm, fill=black] (240:1.75cm) circle (3pt);
\draw[xshift=8.75cm, fill=black] (300:1.75cm) circle (3pt);
    \draw[xshift=8.75cm,thick,line width=2pt] (0:1.75cm) \foreach \x in {60, 120,...,359} {
            -- (\x:1.75cm) 
        } -- cycle (60:1.75cm);
\draw[xshift=8.75cm, fill=black] (0,0) circle (3pt);
    \draw[xshift=8.75cm] (0:1.75cm)--(0,0)--(120:1.75cm) (0,0)--(240:1.75cm);
    \draw[xshift=8.75cm, dashed] (60:1.75cm)--(0,0)--(180:1.75cm) (0,0)--(300:1.75cm);

\draw[xshift=13.125cm, fill=black] (0:1.75cm) circle (3pt);
\draw[xshift=13.125cm, fill=black] (60:1.75cm) circle (3pt);
\draw[xshift=13.125cm, fill=black] (120:1.75cm) circle (3pt);
\draw[xshift=13.125cm, fill=black] (180:1.75cm) circle (3pt);
\draw[xshift=13.125cm, fill=black] (240:1.75cm) circle (3pt);
\draw[xshift=13.125cm, fill=black] (300:1.75cm) circle (3pt);
    \draw[xshift=13.125cm,thick,line width=2pt] (0:1.75cm) \foreach \x in {60, 120,...,359} {
            -- (\x:1.75cm) 
        } -- cycle (60:1.75cm);
    \draw[xshift=13.125cm, dashed] (0:1.75cm)--(120:1.75cm)--(240:1.75cm)--cycle;
    \draw[xshift=13.125cm] (60:1.75cm)--(180:1.75cm)--(300:1.75cm)--cycle;

\draw[xshift=5.8275cm, yshift=-4.375cm, fill=black] (0:1.75cm) circle (3pt);
\draw[xshift=5.8275cm, yshift=-4.375cm, fill=black] (36:1.75cm) circle (3pt);
\draw[xshift=5.8275cm, yshift=-4.375cm, fill=black] (72:1.75cm) circle (3pt);
\draw[xshift=5.8275cm, yshift=-4.375cm, fill=black] (108:1.75cm) circle (3pt);
\draw[xshift=5.8275cm, yshift=-4.375cm, fill=black] (144:1.75cm) circle (3pt);
\draw[xshift=5.8275cm, yshift=-4.375cm, fill=black] (180:1.75cm) circle (3pt);
\draw[xshift=5.8275cm, yshift=-4.375cm, fill=black] (216:1.75cm) circle (3pt);
\draw[xshift=5.8275cm, yshift=-4.375cm, fill=black] (252:1.75cm) circle (3pt);
\draw[xshift=5.8275cm, yshift=-4.375cm, fill=black] (288:1.75cm) circle (3pt);
\draw[xshift=5.8275cm, yshift=-4.375cm, fill=black] (324:1.75cm) circle (3pt);
    \draw[xshift=5.8275cm, yshift=-4.375cm,thick,line width=2pt] (0:1.75cm) \foreach \x in {36, 72,...,359} {
            -- (\x:1.75cm) 
        } -- cycle (36:1.75cm);
\draw[xshift=5.8275cm, yshift=-4.375cm, fill=black] (0:0.875cm) circle (3pt);
\draw[xshift=5.8275cm, yshift=-4.375cm, fill=black] (72:0.875cm) circle (3pt);
\draw[xshift=5.8275cm, yshift=-4.375cm, fill=black] (144:0.875cm) circle (3pt);
\draw[xshift=5.8275cm, yshift=-4.375cm, fill=black] (216:0.875cm) circle (3pt);
\draw[xshift=5.8275cm, yshift=-4.375cm, fill=black] (288:0.875cm) circle (3pt);
    \draw[xshift=5.8275cm, yshift=-4.375cm] (0:0.875cm) \foreach \x in {72, 144, ...,359} {
            -- (\x:0.875cm) 
        } -- cycle (72:0.875cm);
    \draw[xshift=5.8275cm, yshift=-4.375cm] (0:1.75cm)--(0:0.875cm) (72:1.75cm)--(72:0.875cm) (144:1.75cm)--(144:0.875cm) (216:1.75cm)--(216:0.875cm) (288:1.75cm)--(288:0.875cm);
    \draw[xshift=5.8275cm, yshift=-4.375cm, dashed] (36:0.875cm) \foreach \x in {36, 108, ...,359} {
            -- (\x:0.875cm) 
        } -- cycle (36:0.875cm);
    \draw[xshift=5.8275cm, yshift=-4.375cm, dashed] (36:1.75cm)--(36:0.875cm) (108:1.75cm)--(108:0.875cm) (180:1.75cm)--(180:0.875cm) (252:1.75cm)--(252:0.875cm) (324:1.75cm)--(324:0.875cm);

\draw[xshift=11.655cm, yshift=-4.375cm, fill=black] (0:1.75cm) circle (3pt);
\draw[xshift=11.655cm, yshift=-4.375cm, fill=black] (36:1.75cm) circle (3pt);
\draw[xshift=11.655cm, yshift=-4.375cm, fill=black] (72:1.75cm) circle (3pt);
\draw[xshift=11.655cm, yshift=-4.375cm, fill=black] (108:1.75cm) circle (3pt);
\draw[xshift=11.655cm, yshift=-4.375cm, fill=black] (144:1.75cm) circle (3pt);
\draw[xshift=11.655cm, yshift=-4.375cm, fill=black] (180:1.75cm) circle (3pt);
\draw[xshift=11.655cm, yshift=-4.375cm, fill=black] (216:1.75cm) circle (3pt);
\draw[xshift=11.655cm, yshift=-4.375cm, fill=black] (252:1.75cm) circle (3pt);
\draw[xshift=11.655cm, yshift=-4.375cm, fill=black] (288:1.75cm) circle (3pt);
\draw[xshift=11.655cm, yshift=-4.375cm, fill=black] (324:1.75cm) circle (3pt);
    \draw[xshift=11.655cm, yshift=-4.375cm,thick,line width=2pt] (0:1.75cm) \foreach \x in {36, 72,...,359} {
            -- (\x:1.75cm) 
        } -- cycle (36:1.75cm);
\draw[xshift=11.655cm, yshift=-4.375cm, fill=black] (0,0) circle (3pt);
    \draw[xshift=11.655cm, yshift=-4.375cm] (0:1.75cm)--(0,0) (72:1.75cm)--(0,0) (144:1.75cm)--(0,0) (216:1.75cm)--(0,0) (288:1.75cm)--(0,0);
    \draw[xshift=11.655cm, yshift=-4.375cm] (0:1.75cm)--(72:1.75cm)--(144:1.75cm)--(216:1.75cm)--(288:1.75cm)--cycle;
    \draw[xshift=11.655cm, yshift=-4.375cm, dashed] (36:1.75cm)--(0,0) (108:1.75cm)--(0,0) (180:1.75cm)--(0,0) (252:1.75cm)--(0,0) (324:1.75cm)--(0,0);
    \draw[xshift=11.655cm, yshift=-4.375cm, dashed] (36:1.75cm)--(108:1.75cm)--(180:1.75cm)--(252:1.75cm)--(324:1.75cm)--cycle;
\end{tikzpicture}
\captionof{figure}{ }
\end{center}
In non-regular polyhedra, cyclic sequences with the property described above can have self intersections.
For example, a $3$-gonal bipyramid (Fig.2) has only one such sequence (up to reversing).
\begin{center}
\begin{tikzpicture}[scale=0.3]

%\draw[step=0.5cm,color=gray] (-5,-6) grid (5,4);
%\draw[fill=white] (0,0) circle (3pt);

\coordinate (A) at (0.5,4);
\coordinate (B) at (0.5,-5);
\coordinate (1) at (0,-1);
\coordinate (2) at (-2,0);
\coordinate (3) at (3,0);

\draw[fill=black] (A) circle (3.5pt);
\draw[fill=black] (B) circle (3.5pt);

\draw[fill=black] (1) circle (3.5pt);
\draw[fill=black] (2) circle (3.5pt);
\draw[fill=black] (3) circle (3.5pt);

\draw[thick] (3)--(1)--(2);
\draw[thick, dashed] (2)--(3);

\draw[thick] (A)--(1);
\draw[thick] (A)--(2);
\draw[thick] (A)--(3);

\draw[thick] (B)--(1);
\draw[thick] (B)--(2);
\draw[thick] (B)--(3);

\end{tikzpicture}
\captionof{figure}{ }
\end{center}

Analogues of Petrie polygons for {\it maps}, i.e. $2$-cell embeddings of graphs in closed $2$-dimensional surfaces, 
are known as {\it zigzags} \cite{DDS-book,Lins1} or {\it closed left-right paths} \cite{GR-book, Shank}.
General results on zigzags in plane graphs can be found in \cite[Chapter 17]{GR-book}.
A large portion of material given in \cite{DDS-book} concerns zigzags in three-regular plane graphs related to mathematical chemistry,
in particular, well-known {\it fullerenes}.
In both these books, the case when there is a single zigzag is distinguished.
Following \cite{DDS-book}, we call maps satisfying this condition $z$-{\it knotted}. 

A {\it Gauss code} is a word, where each symbol occurs precisely twice \cite[Section 17.4]{GR-book}.
There is a problem related to realizing  Gauss codes as closed curves with simple self-intersections in closed $2$-dimensional  surfaces.
It is not difficult to see that a map is $z$-knotted if and only if there is a zigzag passing trough each edge twice.
The latter is equivalent to the fact that the corresponding medial graph 
(the graph whose vertex set is formed by all edges of the map, 
and two edges are adjacent if they have a common vertex and belong to the same face, see  the bold graph on Fig.3)
is a realization of a certain Gauss code.
We refer \cite[Section 17.7]{GR-book} for the case of plane graphs
and \cite{CrRos, Lins2} for the general case.
\begin{center}
\begin{tikzpicture}[scale=1.2]

\draw[thick] (0.5,0) -- (1,0) -- (3,0) -- (3.5,0);
\draw[thick] (1.5,1) -- (2,1) -- (4,1) -- (4.5,1);
\draw[thick] (0.65,-0.35) -- (1,0) -- (2,1) -- (2.35,1.35);
\draw[thick] (2.65,-0.35) -- (3,0) -- (4,1) -- (4.35,1.35);

\draw[thick] (1.65,1.35) -- (2,1) -- (3,0) -- (3.35,-0.35);

\draw[thick, line width=2pt] (1.15,0.85) -- (1.5,0.5) -- (2,0) -- (2.35,-0.35);

\draw[thick, line width=2pt] (3.35,1.35) -- (3,1) -- (2,0) -- (1.65,-0.35);

\draw[thick, line width=2pt] (2.65,1.35) -- (3,1) -- (3.5,0.5) -- (3.85,0.15);

\draw[thick, line width=2pt] (1,0.5) -- (1.5,0.5) -- (2.5,0.5) -- (3.5,0.5) -- (4,0.5);

\draw[fill=white] (1,0) circle (2pt);
\draw[fill=white] (3,0) circle (2pt);
\draw[fill=white] (2,1) circle (2pt);
\draw[fill=white] (4,1) circle (2pt);

\draw[fill=black] (2,0) circle (2pt);

\draw[fill=black] (1.5,0.5) circle (2pt);
\draw[fill=black] (2.5,0.5) circle (2pt);
\draw[fill=black] (3.5,0.5) circle (2pt);

\draw[fill=black] (3,1) circle (2pt);

\end{tikzpicture}
\captionof{figure}{}
\end{center}

The concept of $z$-knottedness  can be described in terms of ${\mathbb Z}_{2}$-homologies.
First of all, we observe that every zigzag  is also a zigzag in the dual map and conversely.
If our map is not $z$-knotted, then there is a zigzag passing through some edges once.
These edges form a bicycle, i.e. a cycle which is also a cycle in the dual map.
By  \cite[Theorem 17.3.5]{GR-book}, a plane graph is $z$-knotted if and only if it does not contain non-trivial bicycles.
This is equivalent to the fact that the corank of the Laplacian over ${\mathbb Z}_{2}$ is equal to $1$
and well-known Kirchhoff's theorem implies that a plane graph is $z$-knotted if and only if 
the number of spanning trees in this graph is odd \cite[Sections 14.15 and 17.8]{GR-book}.
In the general case, we consider the vector space over ${\mathbb Z}_{2}$ formed by all bicycles
with the symmetric difference as an additive operation. 
For a $z$-knotted map this vector space does not contain boundaries and coboundaries and 
its dimension is equal to the  corresponding Betti number, i.e.  there is a natural one-to-one correspondence 
between bicycles and $1$-dimensional ${\mathbb Z}_{2}$-homologies  \cite[Section 4]{CrRos}.

In this paper, we investigate zigzags in triangulations of closed $2$-dimensional surfaces.
It is well-known that an $n$-gonal bipyramid is a $z$-knotted triangulation of the sphere ${\mathbb S}^{2}$ if $n$ is odd.
Two examples of $z$-knotted fullerenes can be found in \cite[Section 2.3]{DDS-book}
(the duals are $z$-knotted triangulations of ${\mathbb S}^{2}$). 
A large class of  plane $z$-knotted graphs (in particular, $z$-knotted fullerenes) was obtained using computer \cite{DDS-book}. 

Our main result (Theorem \ref{th-main}) states that every triangulation of any closed $2$-dimensional surface 
admits a $z$-knotted shredding.

This statement will be proved in two steps.
First, we introduce the $z$-monodromy $M_{F}$ which acts on the set of oriented edges of a face $F$ in a triangulation.
For two consecutive oriented edges $e_{0}$ and $e$ in $F$ we consider the unique zigzag containing the sequence $e_{0},e$
and define $M_{F}(e)$ as the first (oriented) edge of $F$ which occurs in this zigzag after $e$.
There are precisely seven possibilities for the $z$-monodromies and only four possibilities can be realized in 
$z$-knotted triangulations (Theorem \ref{th-1}).
Examples show that each of these seven possibilities happens in a certain triangulation; 
moreover, the four possibilities corresponding to the $z$-knotted case are realized in $z$-knotted triangulations (Section 5).
If the $z$-monodromies of all faces in a triangulation are of $z$-knotted type, then this triangulation is $z$-knotted
(Theorem \ref{th-2}).
The second step is the gluing lemma concerning zigzags in the connected sum of triangulations (Section 6).
Using this lemma, we replace every face whose $z$-monodromy is not of $z$-knotted type by
some number of faces with the $z$-monodromies of $z$-knotted types and get a $z$-knotted shredding.

Other application of the gluing lemma is a description of all cases when the connected sum of $z$-knotted triangulations
is $z$-knotted (Theorem \ref{th-3}).

\iffalse
The same questions can be considered for triangulations of surfaces whose dimension is greater than 2, as well as, 
for maps with all possible size of faces.
One of the problems  is to describe all possibilities  for the $z$-monodromies and distinguish all $z$-monodromies which realize for the $z$-knotted case.
\fi

\section{Zigzags in triangulations}
We describe some elementary properties of zigzags in triangulations (which also hold for other maps). 

Let $M$ be a connected closed $2$-dimensional surface (not necessarily orientable)
and let $\Gamma$ be a triangulation of $M$, i.e.
a $2$-cell embedding of a connected simple finite graph in $M$ whose faces are triangles \cite[Section 3.1]{MT-book}.
The following properties are immediately consequences of the definition:
(1) every edge is contained in precisely  two distinct faces,
(2) the intersection of two distinct faces is an edge or a vertex or empty.

Two distinct edges are called  {\it adjacent} if there is a face containing them.
Since every face is a triangle, any two adjacent edges have a common vertex.  
Two faces are said to be {\it adjacent} if their intersection is an edge.

A {\it zigzag}  in $\Gamma$ is a sequence of edges $\{e_{i}\}_{i\in {\mathbb N}}$ satisfying the following conditions for every $i\in {\mathbb N}$:
\begin{enumerate}
\item[(Z1)] $e_{i},e_{i+1}$ are adjacent,
\item[(Z2)] the faces containing $e_{i},e_{i+1}$ and $e_{i+1},e_{i+2}$ are distinct 
and the edges $e_{i}$ and $e_{i+2}$ are disjoint. 
\end{enumerate} 
Since $\Gamma$ is finite, 
for every zigzag $Z=\{e_{i}\}_{i\in {\mathbb N}}$
there is a natural number $n>0$ such that $e_{i+n}=e_{i}$ for every $i\in {\mathbb N}$. 
The smallest number $n$ satisfying this condition is called the {\it length} of $Z$.
So, our zigzag is the {\it cyclic sequence} $e_{1},\dots, e_{n}$, where $n$ is the length of $Z$.
We say that $Z$ is {\it edge-simple} if all edges in this cyclic sequence are mutually distinct.
Note that $Z$ can be presented as a cyclic sequence of vertices $v_{1},\dots,v_{n}$,
where $v_{i}$ and $v_{i+1}$ are the vertices belonging to $e_{i}$ for $i<n$
and the edge $e_{n}$ contains $v_{n}$ and $v_{1}$.
The zigzag $Z$ is called {\it simple} if these vertices are mutually distinct.
All zigzags of the Platonic solids are simple (see Fig.1).
It is clear that a simple zigzag is edge-simple, but an edge-simple zigzag is not necessarily simple.

Observe that every zigzag $\{e_{i}\}_{i\in {\mathbb N}}$ is completely determined by any pair of consecutive edges $e_{i},e_{i+1}$.
If $X=\{e_{1},\dots, e_{n}\}$ is a sequence of edges, 
then $X^{-1}$ denotes the reversed sequence $e_{n},\dots,e_{1}$.
If $Z$ is a zigzag, then the same holds for $Z^{-1}$.
If $Z$ contains a sequence $e,e'$, then the sequence $e',e$ is contained in the reversed zigzag $Z^{-1}$.
Sequences of type $e,e',\dots,e',e$ are not contained in zigzags, i.e. a zigzag cannot be reversed to itself.
Indeed, if such a sequence is contained in a zigzag, then this zigzag is a sequence of type
$$e,e',e_{1},e_{2},\dots,e_{m},e_{m},\dots,e_{2},e_{1},e',e,\dots$$
and there are two consecutive edges which are the same, a contradiction. 

We say that $\Gamma$ is $z$-{\it knotted} 
if it contains only one pair of zigzags $Z,Z^{-1}$, in other words, there is a single zigzag up to reversing.

\begin{lemma}\label{l-zk}
If $\Gamma$ is $z$-knotted, then each of the two zigzags passes through every edge twice.
Conversely, if $\Gamma$ contains a zigzag passing through every edge twice, then it is $z$-knotted.
\end{lemma}

This statement can be found in \cite{DDS-book,GR-book},
but some arguments from its proof will be used in what follows.
For this reason, we present this proof below.

\begin{proof}[Proof of Lemma \ref{l-zk}]
Every edge $e$ is contained in precisely two distinct faces $F_{1}$ and $F_{2}$. 
We fix a vertex on $e$ and denote by $e_{i}$, $i=1,2$ the edge in $F_{i}$ containing this fixed vertex and distinct from $e$ (Fig.4).
Let $Z_{i}(e)$,  $i=1,2$ be the zigzag containing the sequence $e,e_{i}$.
\begin{center}
\begin{tikzpicture}[scale=0.8]

\draw[fill=black] (1,0) circle (2.5pt);
\draw[fill=black] (3,0) circle (2.5pt);
\draw[fill=black] (2,2) circle (2.5pt);
\draw[fill=black] (2,-2) circle (2.5pt);

\draw[thick] (2,2) -- (1,0)-- (3,0)-- cycle;
\draw[thick] (2,-2) -- (1,0)-- (3,0)-- cycle;

\node at (1.25,1.1) {$e_{1}$};
\node at (1.25,-1.1) {$e_{2}$};
\node at (2,0.18) {$e$};
\node at (2,0.85) {$F_{1}$};
\node at (2,-0.9) {$F_{2}$};
\end{tikzpicture}
\captionof{figure}{}
\end{center}
If $\Gamma$ is $z$-knotted, i.e. it contains only one pair of zigzags $Z,Z^{-1}$, 
then each of these zigzags coincides with every $Z_{i}(e)$ or its reverse.
This implies that it passes through $e$ twice.

If a zigzag passes through each edge twice, then it coincides with every $Z_{i}(e)$ or its reverse for all edges $e$.
This is possible only in the case when $\Gamma$ is $z$-knotted.
\end{proof}

\section{Main result}
Let $\Gamma$ and $\Gamma'$ be triangulations of (connected closed $2$-dimensional)  surfaces $M$ and $M'$, respectively.
Suppose that $F$ is a face in $\Gamma$ and $F'$ is a face in $\Gamma'$.
By our assumption, $F$ and $F'$ both are homeomorphic to a closed $2$-dimensional disc 
and each of the boundaries $\partial F$ and $\partial F'$ is the sum of three edges. 
Let $g: \partial F \to \partial F'$ be a homeomorphism transferring every vertex of $F$ to a vertex of $F'$, i.e.
if $v_{i}$, $i\in \{1,2,3\}$ are the vertices of $F$, then $v'_{i}=g(v_{i})$, $i\in \{1,2,3\}$ are the vertices of $F'$.
In what follows, such boundary homeomorphisms will be called {\it special}.

We define the {\it connected sum} $\Gamma \#_{g} \Gamma'$.
We remove the interiors of $F$ and $F'$ from $M$ and $M'$ (respectively)
and glue together the boundaries $\partial F$ and $\partial F'$ such that every $v_{i}$ is identified with $v'_{i}$.
We get a triangulation of the connected sum $M\# M'$ which will be denoted by $\Gamma \#_{g} \Gamma'$.
The vertex set of $\Gamma \#_{g} \Gamma'$ is the union of the vertex sets of $\Gamma$ and $\Gamma'$,
where every $v_i$ is identified with $v'_i$, and the edge set is the union of the edge sets of $\Gamma$ and $\Gamma'$, 
where the edge containing $v_i,v_j$ is identified with the edge containing $v'_i,v'_j$.
Every face of $\Gamma \#_{g} \Gamma'$ is a face of $\Gamma$ or $\Gamma'$ different from $F,F'$.

Note that for other special homeomorphism $h: \partial F \to \partial F'$ the graph $\Gamma \#_{h} \Gamma'$ 
is not necessarily isomorphic to $\Gamma \#_{g} \Gamma'$.
See Fig.9 and Fig.10 for non-isomorphic connected sums of $3$-gonal bipyramids.

Using the connected sums of triangulations, we describe the concept of {\it shredding} of a triangulation.
As above, we suppose that $\Gamma$  is a triangulation of a surface $M$.
Let $F_{1},\dots,F_{k}$ be mutually distinct faces of $\Gamma$ and 
let $\Gamma_{1},\dots, \Gamma_{k}$ be triangulations of the sphere ${\mathbb S}^2$.
For each $i\in \{1,\dots,k\}$  we take a face $F'_i$ in $\Gamma_{i}$ and a special homeomorphism $g_{i}:\partial F_{i}\to \partial F'_{i}$.
The connected sum
\begin{equation}\label{eq-sum}
(((\Gamma \#_{g_{1}} \Gamma_{1})\#_{g_{2}} \Gamma_{2})\dots)\#_{g_{k}}\Gamma_{k}
\end{equation}
is a triangulation of $M$, where every $F_i$ is replaced by a triangulation of a $2$-dimen\-sional disc. 
Every triangulation of $M$ obtained from $\Gamma$ in such a way is said to be a {\it shredding} of $\Gamma$. 

\begin{theorem}\label{th-main}
Every triangulation $\Gamma$ of any connected closed $2$-dimensional surface admits a $z$-knotted shredding.
Suppose that $\Gamma$ contains precisely $2m$ zigzags, i.e. $m$ zigzags up to reversing, and $m>1$.
Then there are $z$-knotted triangulations $\Gamma_{1},\dots, \Gamma_{k}$ of the sphere ${\mathbb S}^2$
such that $k\le m-1$ and the connected sum \eqref{eq-sum} is $z$-knotted.
\end{theorem}

\section{$Z$-monodromy}
Let $\Gamma$ be a triangulation and let $F$ be a face in $\Gamma$ whose vertices are denoted by $a,b,c$.
Let also $\Omega(F)$ be the set consisting of all oriented edges of $F$. Then
$$\Omega(F)=\{ab,bc,ca,ac,cb,ba\},$$
where $xy$ is  the edge from $x\in\{a,b,c\}$ to $y\in\{a,b,c\}$.
If $e$ is the edge $xy$, then we write $-e$ for the edge $yx$.

Consider the permutation 
$$D_{F}=(ab,bc,ca)(ac,cb,ba)$$
on the set $\Omega(F)$ (which is the composition of two commuting $3$-cycles).
If $x,y,z$ are three mutually distinct vertices of $F$, then $D_{F}(xy)=yz$.
The equality $D_{F}(e)=e'$ implies that $D_{F}(-e')=-e$.

For any $e\in \Omega(F)$ we take $e_{0}\in \Omega(F)$ such that $D_{F}(e_{0})=e$
and consider the zigzag $Z$ containing the sequence $e_{0},e$ 
(recall that every zigzag is completely determined by any pair of consecutive edges).
If $e'$ is the first element of  $\Omega(F)$ contained in the zigzag $Z$ after $e$,
then we define $M_{F}(e)=e'$, Fig.5.
The transformation $M_{F}$ of $\Omega(F)$ will be called the $z$-{\it monodromy} associated to the face $F$.
\begin{center}
\begin{tikzpicture}[scale=0.8]

%\draw[step=0.5cm,color=gray] (0,-10) grid (10,5);
%\draw[fill=white] (8,0) circle (3pt);
\draw[fill=black] (5.1961524228,2) circle (3pt);
\draw[fill=black] (3.4641016152,-1) circle (3pt);
\draw[fill=black] (6.9282032304,-1) circle (3pt);

\draw [thick, decoration={markings,
mark=at position 0.54 with {\arrow[scale=1.5,>=stealth]{<}}},
postaction={decorate}] (5.1961524228,2) -- (3.4641016152,-1);

\draw [thick, decoration={markings,
mark=at position 0.52 with {\arrow[scale=1.5,>=stealth]{>}}},
postaction={decorate}] (3.4641016152,-1) -- (6.9282032304,-1);

\draw [thick, decoration={markings,
mark=at position 0.54 with {\arrow[scale=1.5,>=stealth]{>}}},
postaction={decorate}] (5.1961524228,2) -- (6.9282032304,-1);

\node at (4,0.65) {$e_0$};
\node at (6.3,0.65) {$e$};
\node at (5.1961524228,-1.5) {$M_F(e)$};

\node at (6.35,-2.2) {$Z$};

%\draw [thick, blue]  (6.1,0.44) -- (6.5,0.2) -- (7.25,-0.5) -- (7.625,-1.25) -- (7.625,-1.75) -- (7, -2.25) -- (6.25,-1.75) -- (6,-1.35) -- (5.2,-1);

\draw [black, thick, dashed, decoration={markings,
mark=at position 0.4 with {\arrow[scale=1.5,>=stealth]{>}}},
postaction={decorate}] plot [smooth] coordinates {(6.05,0.44) (6.5,0.2) (7.9,-0.8) (8.25,-2) (7.25,-2.5) (2,-2.5) (2.5,-1.5) (5.2,-1)};
\end{tikzpicture}
\captionof{figure}{ }
\end{center}

\begin{exmp}\label{exmp-t}{\rm
Let $e_{1},\dots, e_{n}$ be a zigzag of $\Gamma$
(we consider this zigzag as a cyclic sequence of non-oriented edges).
For every $i\le n-1$ we denote by $F_{i}$ the face containing the edges $e_{i}$ and $e_{i+1}$
and we write $F_{n}$ for the face which contains $e_{n}$ and $e_{1}$. 
It is easy to see that this zigzag is edge-simple if and only if the faces $F_{1},\dots,F_{n}$ are mutually distinct.
In this case, we have $M_{F_{i}}(e_{i+1})=e_{i}$ for every $i\le n-1$ and $M_{F_{n}}(e_{1})=e_{n}$,
where $e_{1},\dots, e_{n}$ are considered as the edges oriented according to the direction of the zigzag.
Therefore, if all zigzags of $\Gamma$ are edge-simple, then $M_{F}=(D_{F})^{-1}$ for every face $F$.
Conversely, if the latter equality holds for every face $F$, then for any zigzag $e_{1},\dots, e_{n}$
the corresponding faces $F_{1},\dots,F_{n}$ are mutually distinct which implies that all zigzags are edge-simple.
So, we have $M_{F}=(D_{F})^{-1}$ for every face $F$ if and only if each zigzag of $\Gamma$ is edge-simple.
All zigzags are simple (and consequently edge-simple) in the following triangulations:
\begin{enumerate}
\item[$\bullet$] tetrahedrons, octahedrons and icosahedrons (three Platonic solids whose faces are triangles),
\item[$\bullet$] the torus triangulation obtained from a grid with diagonals,
\item[$\bullet$] the triangulation of the real projective plane presented on Figure 6.
\end{enumerate}
In Example \ref{m5}, we show that all zigzags in a $(2k)$-gonal bipyramid are edge-simple if $k$ is even.
Note that these zigzags are simple only for $k=2$ (the octahedron case). 
}\end{exmp}

%%%%%%%%%%%%%ProjectivePlane%%%%%%%%%
\begin{center}
\begin{tikzpicture}[scale=0.5]
\draw[fill=black](0,0) circle (3.5pt);
\draw[fill=black] (8,0) circle (3.5pt);
\draw[fill=black] (0,2) circle (3.5pt);
\draw[fill=black] (8,2) circle (3.5pt);
\draw[fill=black] (4,2) circle (3.5pt);
\draw[fill=black] (4,4) circle (3.5pt);
\draw[fill=black] (4,-2) circle (3.5pt);
\draw[fill=black] (2,0) circle (3.5pt);
\draw[fill=black] (6,0) circle (3.5pt);

\draw[thick] (0,2)--(4,4)--(8,2)--(8,0)--(4,-2)--(0,0)--cycle;
\draw[thick] (0,2)--(2,0)--(4,2)--(6,0)--(8,2);
\draw[thick] (0,2)--(4,2)--(8,2);
\draw[thick] (0,0)--(2,0)--(6,0)--(8,0);
\draw[thick] (4,2)--(4,4);
\draw[thick] (2,0)--(4,-2)--(6,0);

\node at (4.35,4.25) {$a$};
\node at (8.4,2) {$b$};
\node at (8.4,0) {$c$};
\node at (4.35,-2.25) {$a$};
\node at (-0.4,0) {$b$};
\node at (-0.4,2) {$c$};
\node at (2.1,0.5) {$f$};
\node at (3.9,1.5) {$d$};
\node at (6,0.4) {$e$};

\end{tikzpicture}
\captionof{figure}{ }
\end{center}
%%%%%%%%%%%%%%%%%%%%%%%%%%%%%%%%%%%%%%%%%%%

Denote by ${\mathcal Z}(F)$ the set of all zigzags containing the sequences $e,D_{F}(e)$ with $e\in \Omega(F)$.

\begin{lemma}\label{l-0}
The following assertions are fulfilled:
\begin{enumerate}
\item[(1)] A zigzag belongs to ${\mathcal Z}(F)$ if and only if it contains at least one edge of $F$.
\item[(2)] If a zigzag belongs to ${\mathcal Z}(F)$, then the same holds for the reversed zigzag.
\item[(3)] $|{\mathcal Z}(F)|$ is equal to $2$ or $4$ or $6$.
\end{enumerate}
\end{lemma}

\begin{proof}
(1). Suppose that a zigzag passes through the edge of $F$ containing vertices $x,y\in \{a,b,c\}$ and goes from $x$ to $y$.
Then one of the following possibilities is realized: this zigzag contains the sequence $xy,D_{F}(xy)$,
or it contains the sequence $e,xy$ such that $D_{F}(e)=xy$.
In each of these cases, the zigzag belongs to ${\mathcal Z}(F)$.

(2). If a zigzag belongs to ${\mathcal Z}(F)$, then it contains the sequence $e,e'$, where $e\in \Omega(F)$ and $e'=D_{F}(e)$.
The reversed zigzag contains the sequence $-e',-e$.
Since $D_{F}(-e')=-e$, it belongs to  ${\mathcal Z}(F)$.

(3). The set $\Omega(F)$ consists of $6$ elements, but for some distinct $e,e'\in {\Omega}(F)$
the zigzags containing the sequences $e,D_{F}(e)$ and $e',D_{F}(e')$ can be coincident.
In this case, the reversed zigzags also are coincident.
\end{proof}

We say that $\Gamma$ is {\it locally $z$-knotted} for $F$ if $|{\mathcal Z}(F)|=2$.
By Lemma \ref{l-0}, this holds if and only if  there is a single pair of zigzags $Z,Z^{-1}$ containing edges of $F$.

\begin{lemma}\label{l-zk-loc}
If $\Gamma$ is locally $z$-knotted for $F$, then every zigzag from ${\mathcal Z}(F)$ passes through each edge of $F$ twice.
Conversely, if there is a zigzag passing through each edge of $F$ twice, 
then $\Gamma$ is locally $z$-knotted for $F$.
\end{lemma}

\begin{proof}
Similar to the proof of Lemma \ref{l-zk}.
\end{proof}

The main result of this section is a description of all possibilities for the $z$-monodromy.  

 \begin{theorem}\label{th-1}
For the $z$-monodromy $M_{F}$ one of the following possibilities is realized:
\begin{enumerate}
\item[(M1)] $M_{F}$ is identity,
\item[(M2)] $M_{F}=D_{F}$,
\item[(M3)] $M_{F}=(-e_{1},e_{2},e_{3})(-e_{3},-e_{2},e_{1})$, where $(e_{1},e_{2},e_{3})$ is one of the cycles in the permutation  $D_{F}$,
\item[(M4)] $M_{F}=(e_{1},-e_{2})(e_{2},-e_{1})$, where $(e_{1},e_{2},e_{3})$ is one of the cycles in $D_{F}$
{\rm(}$e_{3}$ and $-e_{3}$ are fixed points{\rm)},
\item[(M5)] $M_{F}=(D_{F})^{-1}$,
\item[(M6)] $M_{F}=(-e_{1},e_{2},e_{3})(-e_{3},-e_{2},e_{1})$, where $(e_{1},e_{2},e_{3})$ is one of the cycles in the permutation $(D_{F})^{-1}$,
\item[(M7)] $M_{F}=(e_{1},e_{2})(-e_{1},-e_{2})$, where $(e_{1},e_{2},e_{3})$ is one of the cycles in $D_{F}$
{\rm(}$e_{3}$ and $-e_{3}$ are fixed points{\rm)}.
\end{enumerate}
The triangulation $\Gamma$ is locally $z$-knotted for $F$ if and only if one of the cases {\rm (M1)--(M4)} is realized.
\end{theorem}

In the next section, we give an example for each of the seven possibilities described in Theorem \ref{th-1}.

\begin{lemma}\label{l-1}
The following assertions are fulfilled: 
\begin{enumerate}
\item[(1)] The equality $M_{F}(e)=e'$ implies that $M_{F}(-e')=-e$.
\item[(2)] $M_{F}$ is bijective.
\item[(3)] $M_{F}(e)\ne -e$ for every $e\in \Omega(F)$.
\item[(4)] The length of every cycle in the permutation $M_{F}$ is not greater than $3$.
\end{enumerate}
\end{lemma}

\begin{proof}
(1). 
Let $e\in \Omega(F)$. Consider $e_{0}\in\Omega(F)$ satisfying $D_{F}(e_{0})=e$.
If $Z$ is the zigzag containing the sequence $e_{0},e$, 
then 
$$e'=M_{F}(e)\;\mbox{ and }\;e'_{0}=D_{F}M_{F}(e)$$ 
are the next two elements of $\Omega(F)$ in this zigzag.
Observe that $D_{F}(-e'_{0})=-e'$.
The reversed zigzag $Z^{-1}$ contains the sequence $-e'_{0},-e'$ and $-e$ is the first element of $\Omega(F)$ contained in $Z^{-1}$ after $-e'$.
This means that $M_{F}(-e')=-e$.

(2).
It is sufficient to show that $M_{F}$ is injective.
Suppose that $M_{F}(e)=M_{F}(e')=e''$.
By (1), we have $-e=M_{F}(-e'')=-e'$ which implies that $e=e'$.

(3). 
Let $e$ and $e_{0}$ be as in the proof of (1).
If $M_{F}(e)=-e$, then the zigzag containing the sequence $e_{0},e$ contains also the sequence  $-e,D_{F}(-e)$.
Since $D_{F}(-e)=-e_{0}$, this zigzag contains $e_{0},e$ together with the reversed sequences $-e,-e_{0}$
which is impossible.

(4).
Suppose that the permutation $M_{F}$ contains a cycle of the length greater than $3$. 
Let $e_1, e_2, e_3\in \Omega(F)$ be consecutive elements in this cycle. Then
$$M_{F}(e_1)=e_2,\;\;M_{F}(e_2)=e_3,\;\;M_{F}(e_{3})\ne e_{1}$$ 
and we have $M_{F}(e_{3})\ne -e_{3}$ by (3).
Therefore, $M_{F}(e_{3})$ is equal to $-e_{1}$ or $-e_{2}$.
By (1), the equality $M_{F}(e_{3})=-e_{2}$ implies that $M_{F}(e_{2})=-e_{3}$ which is impossible.
So, we have $M_{F}(e_{3})=-e_{1}$.
Then $M_{F}(e_{1})=-e_{3}$. The latter means that $e_{2}=-e_{3}$
which contradicts $M_{F}(e_{2})=e_{3}$ by (3).
\end{proof}

Using Lemma \ref{l-1}, we show that $M_{F}$ is one of the permutations (M1)--(M7).
Suppose that $M_{F}$ is not identity.

Consider the case when $M_{F}$ contains a $3$-cycle $C$.
By (3), this cycle does not contain pairs of type $e,-e$.
This implies the existence of $e_{1},e_{2},e_{3}\in \Omega(F)$ such that $(e_{1},e_{2},e_{3})$ is a cycle in $D_{F}$ or $(D_{F})^{-1}$
and 
$$C=(e_{1},e_{2},e_{3})\;\mbox{ or }\;C=(-e_{1},e_{2},e_{3}).$$
Then (1) shows that 
$$M_{F}=(e_{1},e_{2},e_{3})(-e_{3},-e_{2},-e_{1})\;\mbox{ or }\;M_{F}=(-e_{1},e_{2},e_{3})(-e_{3},-e_{2},e_{1}).$$
We get the permutation (M2) or (M5) in the first case and (M3) or (M6) in the second.

If there is no $3$-cycle in $M_{F}$, then it contains a transposition $T$.
By (3), this transposition is not of type $(e,-e)$.
Then there exist $e_{1},e_{2},e_{3}\in \Omega(F)$ such that $(e_{1},e_{2},e_{3})$ is one of the cycles in $D_{F}$  and
$$T=(e_{1},e_{2})\;\mbox{ or }\;T=(e_{1},-e_{2}).$$
It follows from (1)  that $M_{F}$ contains also the transposition $(-e_{1},-e_{2})$ or the transposition $(e_{2},-e_{1})$, respectively.
Then (3) implies that $M_{F}$ leaves fixed $e_{3}$ and $-e_{3}$.
So, we get (M4) or (M7).

\begin{lemma}\label{l-2}
The triangulation $\Gamma$ is locally $z$-knotted for $F$ if and only if $D_{F}M_{F}$ is the composition of two distinct commuting $3$-cycles.
\end{lemma}

\begin{proof}
Let $e,e_{0}\in \Omega(F)$ and $D_{F}(e_{0})=e$.
Consider the zigzag $Z$ containing the sequence $e_{0},e$.
If $e',e'_{0}\in \Omega(F)$, $D_{F}(e'_{0})=e'$ and $Z$ contains the sequence $e'_{0},e'$,
then we write $[e',M_{F}(e')]$ for the part of $Z$ between $e'$ and $M_{F}(e')$.
The zigzag $Z$ is the cyclic sequence 
$$[e,M_{F}(e)],[D_{F}M_{F}(e),M_{F}D_{F}M_{F}(e)],\dots,[(D_{F}M_{F})^{m-1}(e),M_{F}(D_{F}M_{F})^{m-1}(e)],$$ 
where $m$ is the smallest non-zero number satisfying $(D_{F}M_{F})^{m}(e)=e$. 
Since
$$e, D_{F}M_{F}(e),\dots, (D_{F}M_{F})^{m-1}(e)$$
are mutually distinct, the same holds for
$$M_{F}(e),M_{F}D_{F}M_{F}(e),\dots, M_{F}(D_{F}M_{F})^{m-1}(e).$$
Consider  the sets 
$${\mathcal X}=\{e, D_{F}M_{F}(e),\dots, (D_{F}M_{F})^{m-1}(e)\}$$
and 
$${\mathcal Y}=\{-M_{F}(e), -M_{F}D_{F}M_{F}(e),\dots, -M_{F}(D_{F}M_{F})^{m-1}(e)\}.$$
If $e'$ belongs to ${\mathcal X}\cap {\mathcal Y}$ and $D_{F}(e'')=e'$, then $D_{F}(-e')=-e''$ and
$Z$ is a cyclic sequence of type 
$$\dots,[*,e''],[e',*],\dots,[*,-e'],[-e'',*],\dots,$$
in other word, the zigzag $Z$ is self-reversed which is impossible.
So, ${\mathcal X}\cap {\mathcal Y}=\emptyset$.
This implies that $m\le 3$ (indeed, if $m>3$, then ${\mathcal X}$ and ${\mathcal Y}$ both contain more that three elements and 
have a non-empty intersection).
Our zigzag 
$$Z=[e_{1},M_{F}(e_{1})],\dots,[e_{m},M_{F}(e_{m})],\;\;\;\;\;e_{i}=(D_{F}M_{F})^{i-1}(e)$$
corresponds to the $m$-cycle $C=(e_{1},\dots, e_{m})$ in the permutation $D_{F}M_{F}$
and the reversed zigzag 
$$Z^{-1}=[-M_{F}(e_{m}),-e_{m}],\dots,[-M_{F}(e_{1}),-e_{1}]$$
corresponds to the $m$-cycle $$C'=(-M_{F}(e_{m}),\dots,-M_{F}(e_{1}))$$
(if $m=1$, then $e_{1}$ and $-M_{F}(e_{1})$ both are fixed points of $D_{F}M_{F}$).

Suppose that $m=3$. Then $D_{F}M_{F}$ is the composition of the commuting $3$-cycles $C$ and $C'$.
Also, we have $\Omega(F)={\mathcal X}\cup {\mathcal Y}$ and every element of $\Omega(F)$ belongs to ${\mathcal X}$ or ${\mathcal Y}$.
This means that ${\mathcal Z}(F)$ consists of $Z$ and $Z^{-1}$.

If $m<3$, then there are elements of $\Omega(F)$ which do not belong to ${\mathcal X}\cup {\mathcal Y}$.
Such elements define zigzags distinct from $Z$ and $Z^{-1}$.
Observe that $D_{F}M_{F}$ does not contain $3$-cycles in this case.
\end{proof}

A direct verification shows that $D_{F}M_{F}$ is the composition of two distinct commuting $3$-cycles
if and only if $M_{F}$ is one of (M1)--(M4). Theorem \ref{th-1} is proved.

\begin{theorem}\label{th-2}
The triangulation $\Gamma$ is $z$-knotted if and only if for every face $F$ the $z$-monodromy $M_{F}$ is one of {\rm(M1)--(M4)}.
\end{theorem}

\begin{proof}
If $\Gamma$ is $z$-knotted, then it is locally $z$-knotted for every face $F$ and 
each $z$-monodromy $M_{F}$ is one of (M1)--(M4) by Theorem \ref{th-1}.

Now, we suppose that  for every face $F$ the $z$-monodromy $M_{F}$ is one of  (M1)--(M4).
It follows from Theorem \ref{th-1} that $\Gamma$ is locally $z$-knotted for all faces. 
Let $F$ be a face of $\Gamma$ and let ${\mathcal Z}(F)=\{Z,Z^{-1}\}$. 
Then $Z$ passes through each edge of $F$.
Therefore, if $F'$ is a face adjacent to $F$ (i.e. intersecting $F$ in an edge),
then $Z$ belongs to ${\mathcal Z}(F')$ by the statement (1) from Lemma \ref{l-0}.
Since $\Gamma$ is locally $z$-knotted for $F'$, we have ${\mathcal Z}(F')=\{Z,Z^{-1}\}$.
The same holds for every face $F'$ of $\Gamma$ by connectedness.
\end{proof}

\begin{rem}\label{rem1}{\rm
Suppose that the $z$-monodromy of a face $F$ is (M1) or (M2) and $Z\in {\mathcal Z}(F)$.
Then $Z$ passes through each edge of $F$ twice in the same direction,
in other words, it goes through three elements of $\Omega(F)$ twice. 
These elements form a cycle in $D_{F}$, see Fig.7(a).
\begin{center}
\begin{tikzpicture}[scale=0.6]

\draw[fill=black] (0,2) circle (3pt);
\draw[fill=black] (-1.7320508076,-1) circle (3pt);
\draw[fill=black] (1.7320508076,-1) circle (3pt);

\draw [thick, decoration={markings,
mark=at position 0.62 with {\arrow[scale=1.5,>=stealth]{>>}}},
postaction={decorate}] (0,2) -- (-1.7320508076,-1);

\draw [thick, decoration={markings,
mark=at position 0.62 with {\arrow[scale=1.5,>=stealth]{>>}}},
postaction={decorate}] (-1.7320508076,-1) -- (1.7320508076,-1);

\draw [thick, decoration={markings,
mark=at position 0.62 with {\arrow[scale=1.5,>=stealth]{<<}}},
postaction={decorate}] (0,2) -- (1.7320508076,-1);

\node at (0,-1.65) {(a)};

\draw[fill=black] (5.1961524228,2) circle (3pt);
\draw[fill=black] (3.4641016152,-1) circle (3pt);
\draw[fill=black] (6.9282032304,-1) circle (3pt);

\draw [thick, decoration={markings,
mark=at position 0.62 with {\arrow[scale=1.5,>=stealth]{><}}},
postaction={decorate}] (5.1961524228,2) -- (3.4641016152,-1);

\draw [thick, decoration={markings,
mark=at position 0.62 with {\arrow[scale=1.5,>=stealth]{>>}}},
postaction={decorate}] (3.4641016152,-1) -- (6.9282032304,-1);

\draw [thick, decoration={markings,
mark=at position 0.62 with {\arrow[scale=1.5,>=stealth]{><}}},
postaction={decorate}] (5.1961524228,2) -- (6.9282032304,-1);

\node at (5.1961524228,-1.65) {(b)};
\end{tikzpicture}
\captionof{figure}{ }
\end{center}
In the case when the $z$-monodromy $M_{F}$ is (M3) or (M4),
every zigzag from ${\mathcal Z}(F)$ goes through one edge twice in the same direction and through the remaining two edges twice in opposite directions, 
see Fig.7(b).
}\end{rem}

\begin{rem}\label{rem2}{\rm
There is a one-to-one correspondence between cycles of the permutation $D_{F}M_{F}$ and zigzags belonging to ${\mathcal Z}(F)$
(see the proof of Lemma \ref{l-2}). 
An easy verification shows that $|{\mathcal Z}(F)|=6$ if $M_{F}$ is (M5) and we have $|{\mathcal Z}(F)|=4$ if $M_{F}$ is (M6) or (M7). 
}\end{rem}

\section{Examples}
We describe zigzags in bipyramids and their connected sums (note that all zigzags will be presented as sequences of vertices)
and give an example for each type of $z$-monodromy.  
In particular, we show that each of the $z$-monodromies (M1)--(M4) is realized in a $z$-knotted triangulation of the sphere ${\mathbb S}^{2}$.

\subsection{$Z$-monodromy in bipyramids}
Consider the $n$-gonal bipyramid $BP_n$, $n\ge 3$ 
containing an $n$-gone whose vertices  are denoted by $1,\dots,n$ and connected with two disjoint vertices $a, b$
(see Fig.8 for $n=3$).
We describe the $z$-monodromy $M_{F}$ for the face $F$ containing the vertices $a,1,2$.
For other faces the $z$-monodromy is the same by the symmetry.
\begin{center}
\begin{tikzpicture}[scale=0.3]

%\draw[step=0.5cm,color=gray] (-5,-6) grid (5,4);
%\draw[fill=white] (0,0) circle (3pt);

\coordinate (A) at (0.5,4);
\coordinate (B) at (0.5,-5);
\coordinate (1) at (0,-1);
\coordinate (3) at (-2,0);
\coordinate (2) at (3,0);

\draw[fill=black] (A) circle (3.5pt);
\draw[fill=black] (B) circle (3.5pt);

\draw[fill=black] (1) circle (3.5pt);
\draw[fill=black] (3) circle (3.5pt);
\draw[fill=black] (2) circle (3.5pt);

\draw[thick] (2)--(1)--(3);
\draw[thick, dashed] (3)--(2);

\draw[thick] (A)--(1);
\draw[thick] (A)--(2);
\draw[thick] (A)--(3);

\draw[thick] (B)--(1);
\draw[thick] (B)--(2);
\draw[thick] (B)--(3);

\node at (0.5,4.6) {$a$};
\node at (0.5,-5.8) {$b$};

\node at (0.5,-1.5) {$1$};
\node at (-2.6,0) {$3$};
\node at (3.6,0) {$2$};

\end{tikzpicture}
\captionof{figure}{ }
\end{center}
\begin{exmp}[$z$-monodromy of type (M3)]\label{m3}\rm{
Suppose that $n=2k+1$ and $k$ is odd. 
If $k=1$, then one of the zigzags is
$$a,1,2,b,3,1,a,2,3,b,1,2,a,3,1,b,2,3.$$
For $k\ge 3$ this zigzag is
$$a,1,2,b,3,4,\dots,a,n-2,n-1,b,n,1,a,2,3,b,\dots,a,n-1,n,$$
$$b,1,2,a,3,4,\dots,b,n-2,n-1,a,n,1,b,2,3,a,\dots,b,n-1,n.$$
The zigzag passes through every edge twice and, by Lemma \ref{l-zk},
the bipyramid is $z$-knot\-ted.
The face $F$ appears in the zigzag as follows
$$a,1,2,\dots,1,a,2,\dots,1,2,a,\dots$$
which determines $M_{F}$ for three elements of $\Omega(F)$
$$12\rightarrow 1a,\;\; a2\rightarrow 12,\;\; 2a\rightarrow a1.$$
By the statement (1) from Lemma \ref{l-1}, we have
$$a1\rightarrow 21,\;\; 21\rightarrow 2a,\;\; 1a\rightarrow a2.$$
Let $e_{1}=12$, $e_{2}=2a$, $e_{3}=a1$. 
Then $(e_{1},e_{2},e_{3})$ is one of the $3$-cycles in $D_{F}$. 
A simple verification shows that
$$M_{F}=(-e_{1},e_{2},e_{3})(-e_{3},-e_{2},e_{1})$$
is of type (M3).
}\end{exmp}

\begin{exmp}[$z$-monodromy of type (M4)]\label{m4}\rm{
Suppose that $n=2k+1$ and $k$ is even. 
If $k=2$, then  one of the zigzags is
$$a,1,2,b,3,4,a,5,1,b,2,3,a,4,5,
b,1,2,a,3,4,b,5,1,a,2,3,b,4,5.$$
For $k\ge 4$ we have
$$a,1,2,b,3,4,a,\dots,b,n-2,n-1,a,n,1,b,2,3,a,\dots,a,n-1,n,$$
$$b,1,2,a,3,4,b,\dots,a,n-2,n-1,b,n,1,a,2,3,b,\dots,b,n-1,n.$$
As in the previous example, this zigzag passes through every edge twice and the bipyramid is $z$-knotted.
Let $e_{1}=2a$, $e_{2}=a1$, $e_{3}=12$.
Then $(e_{1},e_{2},e_{3})$ is one of the $3$-cycles in $D_{F}$.
The face $F$ appears in the zigzag as follows
$$a,1,2,\dots,1,2,a,\dots,1,a,2,\dots$$
which implies that $M_{F}$ leaves fixed $e_{3}$ and transfers $e_{1}$ to $-e_{2}$ and $-e_{1}$ to $e_{2}$.
Using the statement (1) from Lemma \ref{l-1}, we establish that
$$M_{F}=(e_{1},-e_{2})(e_{2},-e_{1})$$
is of type (M4).
}\end{exmp}

\begin{exmp}[$z$-monodromy of type (M5)]\label{m5}\rm{
Suppose that $n=2k$ and $k$ is even. 
In the case when $k=2$, we get  the octahedron whose zigzags are simple.
Assume that $k\ge 4$.
The set ${\mathcal Z}(F)$ contains precisely $8$ zigzags:
\begin{enumerate}
\item[$\bullet$] $a,1,2,b,3,4,\dots,b,n-1,n$
\item[$\bullet$] $b,1,2,a,3,4,\dots,a,n-1,n$
\item[$\bullet$] $1,a,2,3,b,\dots,a,n-2,n-1,b,n$
\item[$\bullet$] $1,b,2,3,a,\dots,b,n-2,n-1,a,n$
\end{enumerate}
and their reverses. These zigzags are not simple, but they are edge-simple.
By Example \ref{exmp-t}, the $z$-monodromy of every face is of type (M5).
}\end{exmp}

\begin{exmp}[$z$-monodromy of type (M7)]\label{m7}\rm{
Suppose that $n=2k$ and $k$ is an odd number greater than $1$.
Then ${\mathcal Z}(F)$ is formed by the following two zigzags 
$$a,1,2,b,3,4,\dots,a,n-1,n,b,1,2,a,3,4,\dots,b,n-1,n$$
and
$$1,a,2,3,\dots,b,n-2,n-1,a,n,1,b,2,3,\dots,a,n-2,n-1,b,n$$
and their reverses.
As in Example \ref{m4}, we take $e_{1}=2a$, $e_{2}=a1$, $e_{3}=12$.
Then $(e_{1},e_{2},e_{3})$ is one of the $3$-cycles in $D_{F}$.
The face $F$ appears in the zigzags as follows
$$a,1,2,\dots,1,2,a,\dots\;\mbox{ and }\;1,a,2,\dots$$
which implies that $M_{F}$ leaves fixed $e_{3}$ and transfers $e_{1}$ to $e_{2}$ and $-e_{1}$ to $-e_{2}$.
By the statement (1) from Lemma \ref{l-1}, 
$$M_{F}=(e_{1},e_{2})(-e_{1},-e_{2})$$
is of type (M7).
}\end{exmp}

\subsection{$Z$-monodromy in the connected sums of bipyramids}
Let $BP_{n}$ be as in the previous subsection. 
Consider another one bipyramid $BP_{n'}$, where the vertices of the $n'$-gone are denoted by $1',\dots,n'$ and 
we write $a',b'$ for the remaining two vertices.
Let $S$ and $S'$ be the faces of the bipyramids containing the vertices $a,1,2$ and $a',1',2'$, respectively. 
We describe the $z$-monodromies  for some faces in the connected sums $BP_{n} \#_{g} BP_{n'}$,
where $g:\partial S\to \partial S'$ is a special homeomorphism.
One of these sums for $n=n'=3$ is presented on Fig.9.
\begin{center}
\begin{tikzpicture}[scale=1]

%\draw[step=0.5cm,color=gray] (-5,-5) grid (5,5);
%\draw[fill=white] (0,0) circle (3pt);

% A B C - srodkowy trójkąt
\draw[fill=black] (0,1.5) circle (2pt); % A góra
\draw[fill=black] (0,-1) circle (2pt); % B dól
\draw[fill=black] (-0.5,0) circle (2pt); % C srodek

\draw[fill=black] (-2,0.5) circle (2pt); % L
\draw[fill=black] (1.5,0.5) circle (2pt); % P
\draw[fill=black] (-2,-1.5) circle (2pt); % LD
\draw[fill=black] (1.5,-1.5) circle (2pt); % PD

\coordinate (A) at (0,1.5);
\coordinate (B) at (0,-1);
\coordinate (C) at (-0.5,0);
\coordinate (L) at (-2,0.5);
\coordinate (P) at (1.5,0.5);
\coordinate (LD) at (-2,-1.5);
\coordinate (PD) at (1.5,-1.5);

\draw[thick, dashed, line width=1.75pt] (A)--(B);
\draw[thick, line width=1.75pt] (A)--(C); 
\draw[thick, line width=1.75pt] (B)--(C);

\draw[thick] (L)--(A);
\draw[thick, dashed] (L)--(B);
\draw[thick] (L)--(C);

\draw[thick] (P)--(A);
\draw[thick, dashed] (P)--(B);
\draw[thick] (P)--(C);

\draw[thick] (LD)--(B);
\draw[thick] (LD)--(C);
\draw[thick] (LD)--(L);

\draw[thick] (PD)--(B);
\draw[thick] (PD)--(C);
\draw[thick] (PD)--(P);

\end{tikzpicture}
\captionof{figure}{}
\end{center}

\begin{exmp}[$z$-monodromy of type (M2)]\label{m2}\rm{
Suppose that $n=2k+1$ and $n'=2k'+1$, where $k$ and $k'$ are odd.
Let $\Gamma$ be the connected sum $BP_{n} \#_{g} BP_{n'}$, where $g:\partial S\to \partial S'$ satisfies 
$$g(a)=a',\;\;g(1)=1',\;\;g(2)=2'$$
(see Fig.9 for $n=n'=3$). 
The zigzag of $BP_{n}$ considered in Example \ref{m3} can be presented as the cyclic sequence $A,B,C$,
where
$$A=\{1, 2, b, \dots, 1, a\},\;\;B=\{a, 2, \dots, b, 1, 2\},\;\;C=\{2, a,\dots, 1, b, 2, \dots, a, 1\}$$
are parts of the zigzag between two edges of the face $S$.
Note that any two consecutive parts have the same vertex
(for example, the parts $A$ and $B$ are joined in the vertex $a$).
Similarly, one of the two zigzags of $BP_{n'}$ is the cyclic sequence $A',B',C'$, where
$$A'=\{1', 2', \dots, 1', a'\},\;\;B'=\{a', 2', \dots, 1', 2'\},\;\;C'=\{2', a',\dots, a', 1'\}.$$
Consider the cyclic sequence
$$A, C'^{-1}, B, A', C^{-1}, B',$$
where for any two consecutive parts $X,Y$ the last edge from $X$ is identified with the first edge from $Y$.
A direct verification shows that this is a zigzag in $\Gamma$.
This zigzag passes through each edge of $\Gamma$ twice
(since it is obtained from a zigzag of $BP_{n}$ passing  through all edges of $BP_{n}$ twice 
and a zigzag of $BP_{n'}$ satisfying the same condition).
Lemma \ref{l-zk} implies that $\Gamma$ is $z$-knotted.
Let $F$ be the face of $\Gamma$ containing the vertices $b, 1, 2$
and let $e_{1}=12$, $e_{2}=2b$, $e_{3}=b1$.
Since this face appears in the zigzag as follows
$$1, 2, b, \dots, b, 1, 2, \dots, 2, b, 1, \dots,$$
the $z$-monodromy $M_{F}$ contains the $3$-cycle $(e_{1},e_{2},e_{3})$.
By the statement (1) from Lemma \ref{l-1}, 
$$M_{F}=(e_{1},e_{2},e_{3})(-e_{3},-e_{2},-e_{1})=D_{F}$$
is of type (M2).
}\end{exmp}

\begin{exmp}[$z$-monodromy of type (M1)]\label{m1}\rm{
Suppose that $n=2k$ and $n'=2k'$, where $k$ and $k'$ are odd number greater than $1$.
Denote by $\Gamma$ the connected sum $BP_{n} \#_{g} BP_{n'}$, where $g:\partial S\to \partial S'$ satisfies
$$g(a)=2',\;\;g(1)=a',\;\;g(2)=1'.$$ 
By Example \ref{m7},  the set ${\mathcal Z}(S)$ is formed by the zigzags 
$$a,1,2,b,3,4,\dots,a,n-1,n,b,1,2,a,3,4,\dots,b,n-1,n$$
and
$$1,a,2,3,\dots,b,n-2,n-1,a,n,1,b,2,3,\dots,a,n-2,n-1,b,n$$
and their reverses.
The first zigzag is the cyclic sequence $A,B$, where 
$$A=\{1, 2, \dots, 1, 2\}\;\mbox{ and }\;B=\{2, a, 3, \dots, n,a, 1\}$$
are parts of the zigzag between two edges of the face $S$
($A$ is joined with $B$ in the vertex $2$ and $B$ is joined with $A$ in the vertex $1$).
The second zigzag passes once through the edges $a1$ and $a2$, and it does not contain the edge $12$.
We rewrite this zigzag as 
$$C=\{a, 2, 3,\dots, 2, 3, a, \dots, n,1, a\}.$$
Similarly, ${\mathcal Z}(S')$ is formed by the zigzags $A',B'$ and $C'$, where 
$$A'=\{1', 2', \dots, 1', 2'\},\;\;B'=\{2', a', \dots,a', 1'\},\;\;C'=\{a', 2',\dots,1', a'\},$$
and their reverses.
Then
$$A, C'^{-1}, C^{-1}, A', B, B'$$
(as in the previous example, for any two consecutive parts $X,Y$ the last edge from $X$ is identified with the first edge from $Y$)
is a zigzag in $\Gamma$.
This zigzag passes through all edges of $\Gamma$ twice
(indeed, the sequence $A,B,C$ contains every edge of $BP_{n}$ twice and $A',B',C'$ contains every edge of $BP_{n'}$ twice). 
Therefore, $\Gamma$ is $z$-knotted.
Let $F$ be the face of $\Gamma$ containing the  vertices $a, 2, 3$.
This face appears in the zigzag as follows
$$\dots, a, 3, 2, \dots, 3, 2, a, \dots, 2, a, 3, \dots$$
which implies that $M_{F}$ is identity.
}\end{exmp}

\begin{exmp}[$z$-monodromy of type (M6)]\label{m6}\rm{
As in Example \ref{m2}, we suppose that $n=2k+1$ and $n'=2k'+1$, where $k$ and $k'$ are odd.
Consider the connected sum $BP_{n} \#_{g} BP_{n'}$, where $g:\partial S\to \partial S'$ satisfies 
$$g(a)=1',\;\;g(1)=2',\;\;g(2)=a'$$
(see Fig.10 for $n=n'=3$).
Note that this graph is not isomorphic to the graph from Example \ref{m2}.
\begin{center}
\begin{tikzpicture}[scale=1]

%\draw[step=0.5cm,color=gray] (-5,-5) grid (5,5);
%\draw[fill=white] (0,0) circle (3pt);

% A B C - srodkowy trójkąt
\draw[fill=black] (0,1.5) circle (2pt); % A góra
\draw[fill=black] (0,-1) circle (2pt); % B dól
\draw[fill=black] (0.5,0.1) circle (2pt); % C srodek

\draw[fill=black] (-1.5,0.5) circle (2pt); % L
\draw[fill=black] (2,0.5) circle (2pt); % P
\draw[fill=black] (-1.25,-1.25) circle (2pt); % LD
\draw[fill=black] (1.5,2) circle (2pt); % PG

\coordinate (A) at (0,1.5);
\coordinate (B) at (0,-1);
\coordinate (C) at (0.5,0.1);
\coordinate (L) at (-1.5,0.5);
\coordinate (P) at (2,0.5);
\coordinate (LD) at (-1.25,-1.25);
\coordinate (PG) at (1.5,2);

\draw[thick, dashed, line width=1.75pt] (A)--(B);
\draw[thick, line width=1.75pt] (A)--(C); 
\draw[thick, line width=1.75pt] (B)--(C);

\draw[thick] (L)--(A);
\draw[thick, dashed] (L)--(B);
\draw[thick] (L)--(C);

\draw[thick, dashed] (P)--(A);
\draw[thick] (P)--(B);
\draw[thick] (P)--(C);

\draw[thick] (LD)--(B);
\draw[thick] (LD)--(C);
\draw[thick] (LD)--(L);

\draw[thick] (PG)--(A);
\draw[thick] (PG)--(C);
\draw[thick] (PG)--(P);

\end{tikzpicture}
\captionof{figure}{}
\end{center}
Recall that the single zigzags (up to reversing) in $BP_{n}$ and $BP_{n'}$ are the cyclic sequences $A,B,C$ and $A',B',C'$ (respectively), 
where
$$A=\{1, 2, b, \dots, 1, a\},\;\;B=\{a, 2, \dots, b, 1, 2\},\;\;C=\{2, a,\dots, 1, b, 2, \dots, a, 1\},$$
$$A'=\{1', 2', \dots, 1', a'\},\;\;B'=\{a', 2', \dots,1', 2',\},\;\;C'=\{2', a',\dots, a', 1'\}.$$
The cyclic sequences
$$A, B'^{-1}\;\text{ and }\; B, C', C, A'$$
define zigzags in $BP_{n} \#_{g} BP_{n'}$ (the corresponding edges in  consecutive parts are identified).
Let $F$ be the face of the connected sum which contains the vertices $b, 1, 2$
and let $e_{1}=b1$, $e_{2}=2b$, $e_{3}=12$.
Then $(e_{1},e_{2},e_{3})$ is one of the $3$-cycles in $(D_{F})^{-1}$.
The face $F$ appears in the zigzags as follows
$$1, 2, b, \dots\;\mbox{ and }\;\dots, b, 1, 2, \dots, 1, b, 2, \dots$$
which determines $M_{F}$ on three elements of $\Omega(F)$ 
$$e_{2}\rightarrow e_{3},\;\; e_{3}\rightarrow -e_{1},\;\; -e_{2}\rightarrow e_{1}.$$
The statement (1) from Lemma \ref{l-1} implies that
$$-e_{3}\rightarrow -e_{2},\;\; e_{1}\rightarrow -e_{3},\;\; -e_{1}\rightarrow e_{2}.$$
Therefore, 
$$M_{F}=(-e_{1},e_{2},e_{3})(-e_{3},-e_{2},e_{1})$$
is of type (M6).
}\end{exmp}

\section{Gluing Lemma}
Let $\Gamma$ and $\Gamma'$ be triangulations.
Let also $F$ and $F'$ be faces in $\Gamma$ and $\Gamma'$, respectively. 
Every special homeomorphism $g: \partial F \to \partial F'$  induces a bijection between $\Omega(F)$ and $\Omega(F')$ 
which also will be denoted by $g$ and sends every oriented edge $xy$ to the oriented edge $g(x)g(y)$.
The following properties of the bijection $g:\Omega(F)\to \Omega(F')$ are obvious: 
\begin{enumerate}
\item[$\bullet$] $g(-e)=-g(e)$ for every $e\in \Omega(F)$,
\item[$\bullet$] $gD_{F}g^{-1}=D_{F'}$.
\end{enumerate}

A face $S$ in a triangulation is said to be {\it essential} if every zigzag of this triangu\-lation contains an edge from this face,
or equivalently, every zigzag belongs to ${\mathcal Z}(S)$.
It is clear that all faces in $z$-knotted triangulations are essential.
Also, every face in a tetrahedron is essential.

\begin{lemma}\label{l-3}
The following assertions are fulfilled:
\begin{enumerate}
\item[(1)]
Suppose that $F$ and $F'$ are essential faces.
Then the  connected sum $\Gamma \#_{g} \Gamma'$ is $z$-knotted if and only if 
$gM_{F}g^{-1}M_{F'}$ is the composition of two distinct commuting $3$-cycles.
\item[(2)]
Suppose that $\Gamma'$ is $z$-knotted and $gM_{F}g^{-1}M_{F'}$ is the composition of two distinct commuting $3$-cycles.
Then $\Gamma \#_{g} \Gamma'$ contains a zigzag $Z$ such that 
${\mathcal Z}(S)=\{Z,Z^{-1}\}$
for every face $S$ in $\Gamma \#_{g} \Gamma'$ corresponding to a face of $\Gamma'$ distinct from $F'$.
\end{enumerate}
\end{lemma}

\begin{proof}
For every $e\in \Omega(F)$ there is a unique zigzag of $\Gamma$ containing $e$ and $M_{F}(e)$.
As in the proof of Lemma \ref{l-2},
we write $[e,M_{F}(e)]$ for the part of this zigzag between $e$ and $M_{F}(e)$. 
Let ${\mathcal X}$ be the set of all $[e,M_{F}(e)]$, where $e\in \Omega(F)$.
This set consists of $6$ elements, and
for every $X\in {\mathcal X}$ the reversed path $X^{-1}$ also belongs to ${\mathcal X}$
(indeed, for $X=[e,M_{F}(e)]$ we have $X^{-1}=[e',M_{F}(e')]$, where $e'=-M_{F}(e)$).

Similarly, we denote by ${\mathcal X}'$ the set of all $[e,M_{F'}(e)]$, where $e\in \Omega(F')$
(note that the elements of ${\mathcal X}'$ are contained in zigzags  of $\Gamma'$).

Let $e\in \Omega(F')$. Consider  the cyclic sequence
$$[e,M_{F'}(e)],[g^{-1}M_{F'}(e), M_{F}g^{-1}M_{F'}(e)],$$
$$[gM_{F}g^{-1}M_{F'}(e),M_{F'}(gM_{F}g^{-1}M_{F'})(e)],$$
$$\vdots$$
$$[g^{-1}M_{F'}(gM_{F}g^{-1}M_{F'})^{m-1}(e),M_{F}g^{-1}M_{F'}(gM_{F}g^{-1}M_{F'})^{m-1}(e)],$$
where $m$ is the smallest non-zero number satisfying 
$$(gM_{F}g^{-1}M_{F'})^{m}(e)=e.$$
In fact, this is a cyclic sequence
$$X'_{1},X_{1},\dots,X'_{m},X_{m},$$
where all $X_{i}$ belong to ${\mathcal X}$ and every $X'_{i}$ is an element of ${\mathcal X}'$.
As in the examples from the previous section, 
for any two consecutive parts $A,B\in {\mathcal X}\cup{\mathcal X}'$ in this sequence we identify the last edge from $A$ with the first edge from $B$
and get a zigzag in the connected sum $\Gamma \#_{g} \Gamma'$.
We denote this zigzag by $Z(e)$.
We have
$$X_{i}\ne X^{-1}_{j}\;\mbox{ and }\;X'_{i}\ne X'^{-1}_{j}$$  
for any pair $i,j\in \{1,\dots,m\}$ (otherwise, the zigzag $Z(e)$ is self-reversed which is impossible).
This implies that $m\le 3$.

The zigzag $Z(e)$ corresponds to an $m$-cycle in the permutation $gM_{F}g^{-1}M_{F'}$
and the reversed zigzag $Z(e)^{-1}$ is the cyclic sequence 
$$X^{-1}_{m},X'^{-1}_{m},\dots,X^{-1}_{1},X'^{-1}_{1}$$
related to a different $m$-cycle in this permutation.
Note that $Z(e)^{-1}$ coincides with $Z(e')$ for a certain $e'\in \Omega(F')$.
As in the proof of Lemma \ref{l-2}, we establish that the following two conditions are equivalent:
\begin{enumerate}
\item[(A)] the permutation $gM_{F}g^{-1}M_{F'}$ is the composition of two distinct commuting $3$-cycles,
\item[(B)] for any $e,e'\in \Omega(F')$ the zigzag $Z(e')$ coincides with $Z(e)$ or $Z(e)^{-1}$.
\end{enumerate}

Now, we prove the statements (1) and (2).

(1). If the faces $F$ and $F'$ both are essential, 
then every zigzag of $\Gamma$ or $\Gamma'$ contains an edge from $F$ or $F'$ (respectively).
This implies that every zigzag in the connected sum $\Gamma \#_{g} \Gamma'$
is of type $Z(e)$, where $e\in \Omega(F')$. 
Then (B) will be equivalent to the $z$-knottedness of the connected sum.

(2). Suppose that $\Gamma'$ is $z$-knotted and the condition (A) holds.
Then (B) also holds true and we assert that the zigzag $Z=Z(e)$, $e\in\Omega(F')$ is as required.

The triangulation $\Gamma'$ contains a single pair of zigzags $Z',Z'^{-1}$.
As in the proof of Lemma \ref{l-2}, we have
$$Z'=X'_{1},X'_{2},X'_{3}\;\mbox{ and }\; Z'^{-1}=X'^{-1}_{3},X'^{-1}_{2},X'^{-1}_{1},$$
where $X'_{1},X'_{2},X'_{3},X'^{-1}_{1},X'^{-1}_{2},X'^{-1}_{3}$ are mutually distinct elements of ${\mathcal X}'$,
in other words, the set ${\mathcal X}'$ is formed by these elements.
The condition (B) guarantees that every $X'\in{\mathcal X}'$ is contained in $Z$ or $Z^{-1}$.

By Lemma \ref{l-zk},
each of the zigzags $Z',Z'^{-1}$ passes through every edge of $\Gamma'$ twice.
This implies that each of the zigzags $Z,Z^{-1}$ also passes through every edge of $\Gamma'$ twice.
Using arguments from the proof of  Lemma \ref{l-zk}, 
we show that there is no other zigzag of $\Gamma\#_{g} \Gamma'$ containing edges from $\Gamma'$. 
Therefore, ${\mathcal Z}(S)=\{Z,Z^{-1}\}$ 
for every face $S\ne F'$ of $\Gamma'$ considered as a face of the connected sum $\Gamma\#_{g} \Gamma'$.
\end{proof}

\section{Proof of Theorem \ref{th-main}}
Let $\Gamma$ be a triangulation which is not locally $z$-knotted for a certain face $F$.

\begin{lemma}\label{l-4}
There is a $z$-knotted  triangulation $\Gamma'$ of the sphere ${\mathbb S}^{2}$ and a face $F'$ in this triangulation such that 
$gM_{F}g^{-1}M_{F'}$ is the composition of two distinct commuting 3-cycles for a certain special homeomorphism $g:\partial F\to\partial F'$.
\end{lemma}

\begin{proof}
By Theorem \ref{th-1}, the $z$-monodromy $M_{F}$ is one of (M5)--(M7).

Suppose that $M_{F}$ is (M5) or (M6). Then it is the composition of two distinct commuting $3$-cycles.
Let us  take a $z$-knotted triangulation $\Gamma'$ of ${\mathbb S}^{2}$ containing a face $F'$ such that $M_{F'}$ is identity,
for example, the triangulation from Example \ref{m1}.
Then 
$$gM_{F}g^{-1}M_{F'}=gM_{F}g^{-1}$$ 
is the composition of two distinct commuting 3-cycles for every special homeomorphism $g:\partial F\to\partial F'$.

Consider the case when $M_{F}$ is (M7), i.e. 
$$M_{F}=(e_{1},e_{2})(-e_{1},-e_{2}),$$ 
where $(e_{1},e_{2},e_{3})$ is one of the cycles in $D_{F}$ and $e_{3},-e_{3}$ are fixed points.
Let $\Gamma'$ be the bipyramid $BP_{2k+1}$, where $k$ is odd. 
For every face $F'$ of $\Gamma'$ we have
$$M_{F'}=(-e'_{1},e'_{2},e'_{3})(-e'_{3},-e'_{2},e'_{1}),$$ 
where $(e'_{1},e'_{2},e'_{3})$ is one of the cycles in the permutation  $D_{F'}$, see Example \ref{m3}.
There is a special homeomorphism $g:\partial F\to\partial F'$  satisfying $g(e_{i})=e'_{i}$ for every $i$.
A direct verification shows that $gM_{F}g^{-1}M_{F'}$ is the composition of two distinct commuting 3-cycles.
\end{proof}

Let $\Gamma'$, $F'$ and $g:\partial F\to\partial F'$ be as in Lemma \ref{l-4}.
By the statement (2) from Lemma \ref{l-3}, 
the connected sum $\Gamma \#_{g} \Gamma'$ contains a zigzag $Z$ 
such that ${\mathcal Z}(S)=\{Z,Z^{-1}\}$
for every face $S$ in $\Gamma \#_{g} \Gamma'$ corresponding to a face of $\Gamma'$ distinct from $F'$.
Therefore, $\Gamma \#_{g} \Gamma'$ is locally $z$-knotted for all faces of $\Gamma'$ distinct from $F'$.

\begin{lemma}\label{l-5}
Suppose that $S$ is a face of $\Gamma$ distinct from $F$.
If $\Gamma$ is locally $z$-knotted for $S$, then $\Gamma \#_{g} \Gamma'$ also is locally $z$-knotted for $S$.
\end{lemma}

\begin{proof}
Let $Z_{S}, (Z_{S})^{-1}$ be a unique pair of zigzags in $\Gamma$ containing edges of $S$.
By Lemma \ref{l-zk-loc}, each of these zigzags passes through every edge of $S$ twice.

If  $Z_{S}$ and $(Z_{S})^{-1}$ do not contain any edge from $F$, then they are zigzags in $\Gamma \#_{g} \Gamma'$,
and  the connected sum is locally $z$-knotted for $S$ by Lemma \ref{l-zk-loc}.

If  $Z_{S}$ and $(Z_{S})^{-1}$ contain some edges of $F$, then they belong to ${\mathcal Z}(F)$.
In the proof of Lemma \ref{l-3}, we consider the set ${\mathcal X}$ formed by 
all parts of the zigzags from ${\mathcal Z}(F)$ between $e\in \Omega(F)$ and $M_{F}(e)$.
Recall that every $X\in {\mathcal X}$ is contained in $Z$ or $Z^{-1}$.
Since  $Z_{S}$ and $(Z_{S})^{-1}$ belong to ${\mathcal Z}(F)$, 
they are cyclic sequences formed by at most three  elements of ${\mathcal X}$.
Each of $Z_{S}, (Z_{S})^{-1}$ passes through every edge of $S$ twice and  the same holds for $Z$ and $Z^{-1}$.
Lemma \ref{l-zk-loc} implies that $\Gamma \#_{g} \Gamma'$ is locally $z$-knotted for $S$.
\end{proof}

So, the connected sum $\Gamma \#_{g} \Gamma'$ has the following properties:
\begin{enumerate}
\item[(1)] it is locally $z$-knotted for all faces of $\Gamma'$ distinct from $F'$,
\item[(2)] if $\Gamma$ is locally $z$-knotted for a face $S\ne F$, 
then $\Gamma \#_{g} \Gamma'$ is locally $z$-knotted for $S$.
\end{enumerate}
If $\Gamma \#_{g} \Gamma'$ is not locally $z$-knotted for a face $T$, 
then $T$ is a face of $\Gamma$ by (1).
We apply the above arguments to the face $T$ in $\Gamma \#_{g} \Gamma'$.
Using (2), we construct recursively a shredding of $\Gamma$ which is locally $z$-knotted for every face.
By Theorem \ref{th-2}, this shredding is $z$-knotted.

Suppose that $\Gamma$ contains precisely $2m$ zigzags, i.e. $m$ zigzags up to reversing.
If $\Gamma$ is not locally $z$-knotted for a face $F$, then ${\mathcal Z}(F)$ contains $2$ or $3$ zigzags up to reversing
(Remark \ref{rem2}).
Using the corresponding connected sum, we replace these zigzags by one zigzag
and come to a triangulation with $m-1$ or $m-2$ zigzags up to reversing. 
So, we need at most $m-1$ times to get a $z$-knotted shredding of $\Gamma$.

\section{Connected sums of z-knotted triangulations}

By Theorem \ref{th-1}, there are precisely four types of faces in $z$-knotted triangulations.
The corresponding $z$-monodromies are (M1)--(M4).
In \cite{PT}, these four types were described without $z$-monodromy.
The main result of \cite{PT} determines all cases when the connected sum of two $z$-knotted triangulations is $z$-knotted.
The proof given in \cite{PT} was a long case-by-case verification. 
Now, this result will be presented as a simple consequence of  the statement (1) from Lemma \ref{l-3}.

\begin{theorem}\label{th-3}
Let $\Gamma$ and $\Gamma'$ be $z$-knotted triangulations.
Then the following assertions are fulfilled: 
\begin{enumerate}
\item[{\rm (1)}] 
If $F$ is a face in $\Gamma$ such that $M_{F}=D_{F}$ {\rm(}type {\rm (M2)}{\rm)}, 
then for every face $F'$ in $\Gamma'$ and every special homeomorphism $g: \partial F \to \partial F'$ 
the connected sum $\Gamma \#_{g} \Gamma'$ is $z$-knotted.
\item[{\rm (2)}]
Suppose that $F$ is a face in $\Gamma$ and $M_{F}$ is identity {\rm(}type {\rm (M1)}{\rm)}.
If $F'$ is a face in $\Gamma'$ such that
the connected sum $\Gamma \#_{g} \Gamma'$ is $z$-knotted for a certain special homeomorphism $g: \partial F \to \partial F'$,
then $M_{F'}$ is $D_{F'}$ or {\rm (M3)}. 
In these cases, the connected sum $\Gamma \#_{g} \Gamma'$ is $z$-knotted for every special homeomorphism $g: \partial F \to \partial F'$.
\item[{\rm(3)}] 
If $F$ is a face in $\Gamma$, $F'$ is a face in $\Gamma'$ and $M_{F},M_{F'}$ are {\rm (M3)} or {\rm (M4)},
then there is a special homeomorphism $g: \partial F \to \partial F'$ such that 
the connected sum $\Gamma \#_{g} \Gamma'$ is $z$-knotted.
\end{enumerate}
\end{theorem}

\begin{proof}
Recall that every special homeomorphism $g: \partial F \to \partial F'$  induces a bijection of $\Omega(F)$ to $\Omega(F')$. 
This bijection is also denoted by $g$.

(1). It is easy to see that $gD_{F}g^{-1}=D_{F'}$. 
Therefore, $M_{F}=D_{F}$ implies that 
$$gM_{F}g^{-1}M_{F'}=D_{F'}M_{F'}$$
for every special homeomorphism $g: \partial F \to \partial F'$.
Since $\Gamma'$ is $z$-knotted, $D_{F'}M_{F'}$ is the composition of two distinct commuting $3$-cycles (Lemma \ref{l-2}).
The statement (1) from Lemma \ref{l-3} implies that the connected sum $\Gamma\#_{g}\Gamma'$ is $z$-knotted.

(2). If $M_{F}$ is identity, then 
$$gM_{F}g^{-1}M_{F'}=M_{F'}$$
for every special homeomorphism $g: \partial F \to \partial F'$. 
We get the composition of two distinct commuting $3$-cycles only in the case when $M_{F'}$ is $D_{F'}$ or (M3).
%The statement (1) from Lemma \ref{l-3} gives the claim.

(3).
Suppose that $M_{F}$ and $M_{F'}$ both are (M3), i.e. 
$$M_{F}=(-e_{1},e_{2},e_{3})(-e_{3},-e_{2},e_{1})\;\mbox{ and }\; M_{F'}=(-e'_{1},e'_{2},e'_{3})(-e'_{3},-e'_{2},e'_{1}),$$
where $(e_{1},e_{2},e_{3})$ and $(e'_{1},e'_{2},e'_{3})$ are $3$-cycles in $D_{F}$ and $D_{F'}$, respectively.
There exists a special homeomorphism $g: \partial F \to \partial F'$ which transfers 
$e_{1},e_{2},e_{3}$ to $e'_{1},e'_{2},e'_{3}$ (respectively).
Then
$gM_{F}g^{-1}=M_{F'}$
and
$$gM_{F}g^{-1}M_{F'}=(M_{F'})^2=(M_{F'})^{-1}$$
is the composition of two distinct commuting $3$-cycles.

Consider the case when $F$ and $F'$ both are (M4), i.e.
$$M_{F}=(e_{1},-e_{2})(e_{2},-e_{1})\;\mbox{ and }\; M_{F'}=(e'_{1},-e'_{2})(e'_{2},-e'_{1})$$
where $(e_{1},e_{2},e_{3})$ and $(e'_{1},e'_{2},e'_{3})$ are $3$-cycles in $D_{F}$ and $D_{F'}$, respectively.
We take a special homeomorphism $g: \partial F \to \partial F'$ transferring 
$e_{1},e_{2},e_{3}$ to $e'_{3},e'_{1},e'_{2}$ (respectively).
Then $$gM_{F}g^{-1}=(e'_{1},-e'_{3})(e'_{3},-e'_{1})$$
and 
$$gM_{F}g^{-1}M_{F'}=(-e'_{1},e'_{2},e'_{3})(e'_{1},-e'_{2},-e'_{3})$$
is the composition of two distinct commuting $3$-cycles.

Suppose that $F$ is (M4) and $F'$ is (M3), i.e.
$$M_{F}=(e_{1},-e_{2})(e_{2},-e_{1})\;\mbox{ and }\; M_{F'}=(-e'_{1},e'_{2},e'_{3})(-e'_{3},-e'_{2},e'_{1}),$$
where $(e_{1},e_{2},e_{3})$ and $(e'_{1},e'_{2},e'_{3})$ are $3$-cycles in $D_{F}$ and $D_{F'}$, respectively.
There is a special homeomorphism $g: \partial F \to \partial F'$ sending 
$e_{1},e_{2},e_{3}$ to $e'_{2},e'_{3},e'_{1}$ (respectively).
Then
$$gM_{F}g^{-1}=(e'_{2}, -e'_{3})(e'_{3}, -e'_{2})$$
and  
$$gM_{F}g^{-1}M_{F'}=(e'_{1},e'_{2},-e'_{2})(-e'_{1},-e'_{3},e'_{3})$$
is the composition of two distinct commuting $3$-cycles.

In each of these three cases, the connected sum $\Gamma\#_{g}\Gamma'$ is $z$-knotted by the statement (1) from Lemma \ref{l-3}.
\end{proof}

\end{document}